\documentclass[10pt]{article}
\usepackage{amsmath}
\usepackage{amssymb}
\usepackage{amsthm}
\usepackage{amsfonts}
\usepackage[utf8]{inputenc}
\usepackage{tikz,pgfplots}

\usepackage{fullpage}
\usepackage{tabularx}
\usepackage{multirow}
\usepackage{rotating}
\usepackage{enumerate}

\newtheorem{defn}{Definition}[section]

\newtheorem{lemma}[defn]{Lemma}

\newtheorem{thm}[defn]{Theorem}

\newcommand{\sbigotimes}{%
  \mathop{\mathchoice{\textstyle\bigotimes}{\bigotimes}{\bigotimes}{\bigotimes}}%
}

\usepackage{graphicx}
\usepackage{amsmath}
\usepackage{exscale,times}
\usepackage{float}
\usepackage{color}
\usepackage{setspace}
\usepackage{amssymb}
\usepackage{graphicx}
\usepackage{amsmath}
\usepackage{exscale,times}
\usepackage{epsfig}
\usepackage{subfig}
\usepackage{float}
\usepackage{color}
\usepackage{setspace}
\usepackage{mathtools}
\usepackage{bm}
\usepackage[ruled,vlined]{algorithm2e}
\usepackage[english]{babel}

\title{Low-rank solutions to the stochastic Helmholtz equation}%
\author{Adem Kaya, Melina A. Freitag, \thanks{Institut f\"{u}r Mathematik, Universit\"{a}t Potsdam, 
Campus Golm, Haus 9, Karl-Liebknecht-Str. 24-25, D-14476 Potsdam, Germany ({\tt
  {kaya,melina.freitag}@uni-potsdam.de})}}
  \date{}
\begin{document}
\maketitle

\begin{abstract}    
In this paper, we consider  low-rank approximations for the solutions to the stochastic Helmholtz equation with random coefficients. A Stochastic Galerkin finite element method is used for the discretization of the Helmholtz problem.   Existence theory for the low-rank approximation is established when the system matrix is indefinite.  The low-rank algorithm  does not require the construction of a large system matrix which results in an advantage in terms of CPU time and storage.  Numerical results show that, when the operations in a low-rank method are performed efficiently, it is possible to obtain an advantage in terms of storage and CPU time compared to computations in full rank. We also propose a general approach to implement a preconditioner using the low-rank format efficiently. 
  
\end{abstract}




\section{ Introduction}

Stochastic partial differential equations (SPDE's) are usually concerned with stochastic processes.  
There are three competing methods
 in the literature to solve PDE's with random coefficients:  Monte Carlo method (MCM) \cite{MR2401592,MR2835612},  stochastic collocation methods (SCM) \cite{MR2318799} and  stochastic Galerkin finite element methods (SGFEM) \cite{MR1083354, MR3308418}.
 Different models  and solution methods  for the stochastic Helmholtz equation have been considered in literature \cite{MR2116603,MR4191625,MR2444727,MR2541285}. In \cite{MR2116603}, the stochastic Helmholtz equation with random source function was considered and a multigrid algorithm was applied to a linear system with multiple right-hand sides. In \cite{MR4191625}, an inverse random source problem was considered for the one-dimensional stochastic Helmholtz equation. A Stochastic Helmholtz equation driven by
white noise forcing term in two- and three-dimensions was investigated in \cite{MR2444727}.   In \cite{MR2541285}, stochastic Helmholtz equation with random coefficients was considered, and a parallel Schwarz type domain decomposition preconditioned
recycling Krylov subspace method was applied.

 Solving  (\ref{eqn:stocmain}) is challenging even for  the deterministic  problem. The randomness adds to the difficulty.   When stochastic Galerkin finite element method (SGFEM) is used for the discretization, the randomness in the solution and problem parameters in (\ref{eqn:stocmain}), leads to large linear system of equations. Moreover,  large wave numbers lead to larger linear systems when a conventional discretization scheme is used for the system of deterministic Helmholtz equations due to the pollution error. This requires more memory and CPU time. Additionally, the discretization matrix is indefinite and  standard iterative methods suffer from slow convergence or may fail \cite{MR3050918}.                

In this article, we discretize the Helmholtz problem  (\ref{eqn:stocmain}) with a stochastic Galerkin finite element method combined with a nonstandard finite element  method \cite{adem} to mitigate the pollution error. The resulting system matrix   is very large, sparse and indefinite. We use a preconditioned low-rank Krylov method so that there is no need to store the large system matrix. A candidate is the precondioned low-rank BICG (P-LRBICG) method. Low-rank approaches have been considered for different kinds of stochastic PDEs and with different types of Krylov solvers  \cite{8406fb78f,MR3716585,MR2914549,MR3376791,MR4073359,MR4025775,ciloglu2020stochastic}.

Often, the most time consuming step in a plain preconditioned Krylov method is the implementation of the preconditioner. Matrix vector multiplication, vector vector addition and inner product are relatively cheap in comparison with the implementation of a preconditioner, in particular when dealing with sparse matrices. On the other hand, as we will see, in a low-rank approach, truncation operator, matrix vector multiplication and  trace operator (inner product) may not be so cheap. In order to gain advantage in terms of CPU time, they must be applied appropriately depending on the problem. In this article, we propose a general approach using a low-rank format to implement the preconditioner efficiently for the discretization of the  stochastic PDEs obtained by SGFEM. We also redefine matrix vector multiplication so that it is still cheap enough when the rank is not small.

  The rest of the paper is organized as follows: We introduce the SGFEM in the next section. Existence theory for the low-rank solution for indefinite systems is discussed in Section \ref{sec:existence}.  We introduce the low-rank algorithm and implementation of the preconditioner for the stochastic Helmholtz equation in Section \ref{sec:lowrank}. We provide numerical experiments  which support our theory in Section \ref{sec:numerical} and finish with concluding remarks in Section \ref{sec:conclusion}.

\section{Stochastic Galerkin finite element method (SGFEM)} 
 \label{sec:sgfem}
 
In this paper, we consider the  stochastic Helmholtz equation with random coefficients
\begin{equation}
\left.
\begin{aligned}
-\nabla.(\alpha(\textbf{x},\omega) \nabla u(\textbf{x},\omega))  - c^2 \alpha(\textbf{x},\omega) u(\textbf{x},\omega) &= f(\textbf{x}), \quad \text{in} \ D \times \Omega, \\
u(\textbf{x},\omega) &= g(\textbf{x}), \quad  \textbf{x} \in \partial D_D,\  \omega \in \Omega,\\
\partial_{\nu}u(\textbf{x},\omega) + icu(\textbf{x},\omega) &= 0,  \quad  \textbf{x} \in \partial D_R, \  \omega \in \Omega,
\end{aligned}
 \right\}
 \label{eqn:stocmain}   
\end{equation}
where  $c \in \mathbb{R} $ is the wavenumber,   $\textbf{x} \in D$  is the spatial variable, $f:D\rightarrow \mathbb{R}$ is the source term and $g:\partial D_D \rightarrow \mathbb{R}$ is a given function for the Dirichlet boundary condition. In the  above model, $\alpha(\textbf{x}, \omega)$,
and hence the solution $u(\textbf{x}, \omega)$ are random fields.  We use the same coefficient  $\alpha(\textbf{x}, \omega)$ inside and outside of the divergence so that one can get indefinite linear systems easily for large wavenumbers. We note that such a kind of assumption might move away the problem in (\ref{eqn:stocmain}) from real life applications but the model problem in (\ref{eqn:stocmain}) is useful to explain the techniques in this paper. $\Omega$ is the sample space of events and  $D \subset \mathbb{R}^2$ is a convex bounded polygonal domain where Dirichlet boundary conditions are imposed on $\partial D_D$ and   Robin boundary conditions are imposed on $\partial D_R$. Here,
$\nu$ denotes the unit outward normal to $\partial D$, and $\partial_\nu$ stands for the normal derivative
of $u$. We assume that the random field $\alpha$ is $\mathcal{P}$-almost surely uniformly positive; that is, 
\begin{equation*}
 \exists \beta_1, \beta_2 \quad \text{such that} \quad 0<\beta_1 \leq \beta_2 <+\infty
\end{equation*}
with 
\begin{equation}
 \beta_1 \leq \alpha(\textbf{x}, \omega)  \leq \beta_2, \quad \text{a.e. in} \quad  D \times \Omega.
\label{eqn:positivity}  
\end{equation}
Existence and uniquness of the solution for the stochastic  Helmholtz problem were discussed in \cite{MR4049407}.

We consider the SGFEM for the discretization of (\ref{eqn:stocmain}). The 
SGFEM is widely used for the discretization of the partial differential equations with random coefficients \cite{MR3376791, MR3308418,MR2491431,MR2397299}. It basically consists of  four steps; first of all, the
randomness in the model is  represented with a finite number of random
variables. Then, the random
and spatial dependencies in the random field $\alpha(\textbf{x},\omega)$ are decoupled using the  Karhunen-L\`{o}eve expansion (KLE). In the third stage,  the solution is approximated by a finite-term expansion using a basis of orthogonal polynomials, the so-called generalized polynomial chaos expansion (PCE). Finally, we perform a
Galerkin projection on the set of polynomial basis functions. After applying this procedure, we obtain a coupled system of deterministic Helmholtz equation. For its discretization, we use a stabilized finite element method as proposed in \cite{adem}.  

\subsection{Karhunen-L\`{o}eve expansion (KLE)}

 Let $\alpha:D\times \Omega \rightarrow \mathbb{R}$ be a random field with continuous covariance function $C_{\alpha}(\textbf{x},\textbf{y})$. 
 Then $\alpha$ admits a proper orthogonal decomposition
 \begin{equation}
  \alpha(\textbf{x},\omega) = \bar{\alpha}(\textbf{x}) + \sigma_\alpha \sum_{i=1}^{\infty}\sqrt{\lambda_i}\varphi_i(\textbf{x})\xi_i(\omega), 
 \label{eqn:klei}
 \end{equation}
where $\sigma_\alpha$ is the standard deviation for $\alpha(\textbf{x})$ and $\bar{\alpha}(\textbf{x})$ is the mean.  The random varibles $\xi_i$ are given by 
\begin{equation*}
 \xi_i(\omega) = \frac{1}{\sigma_\alpha \sqrt{\lambda_i}}\int_{D}(\alpha(\textbf{x},\omega)-\bar{\alpha}(\textbf{x})) \varphi_i(\textbf{x})d\textbf{x},
\end{equation*}
and $\{\lambda_i, \varphi_i \}$ is the set of eigenvalues and eigenfunctions of $C_\alpha$ with 
\begin{equation*}
 \int_{D} C_{\alpha}(\textbf{x},\textbf{y})\varphi_i(\textbf{y})d\textbf{y}= \lambda_i \varphi_i(\textbf{x}).
\end{equation*}
In (\ref{eqn:klei}), the eigenvalues,  
${\lambda_i }\rightarrow 0$ as $i \rightarrow \infty$. In practice,
the series is truncated  based on the  decay of the
eigenvalues, and $\alpha$ is approximated by $\alpha_N$  \cite{MR3308418, MR3376791}: 
\begin{equation}
  \alpha_N(\textbf{x},\omega) = \bar{\alpha}(\textbf{x}) + \sigma_\alpha \sum_{i=1}^{N}\sqrt{\lambda_i}\varphi_i(\textbf{x})\xi_i(\omega). 
\label{eqn:kletruncated}
\end{equation}
We have to ensure that $\alpha_N$ satisfies the positivity condition (\ref{eqn:positivity}). 
  For some random inputs, the covariance  functions and eigenpairs  can be computed
explicitly and the positivity condition is satisfied \cite{MR3308418}, otherwise, they can be approximated numerically \cite{MR2105161}.

 \subsection{Generalized polynomial chaos expansion}

For a random field $u\in L^2(\Omega, \mathcal{F}, \mathcal{P})$, we have the expansion 
\begin{equation}
 u(\textbf{x},\omega) = \sum_{j=0}^{\infty} u_j(\textbf{x}) \psi_j(\boldsymbol{\xi}(\omega))
 \label{eqn:infinitesum}
\end{equation}
where $u_j$ are the deterministic modes given by
\begin{equation*}
 u_j(\textbf{x}) = \frac{\langle u(\textbf{x},\omega)\psi_j(\boldsymbol{\xi}(\omega))  \rangle}{\langle \psi_j^2(\boldsymbol{\xi}(\omega)) \rangle},
\end{equation*}
 where  \boldsymbol{$\xi$}$=[ \xi_1,...,\xi_N ]^T$   is a finite-dimensional random vector, $\psi_j$ are multivariate orthogonal
polynomials, $\langle \psi_0(\boldsymbol{\xi}) \rangle=1$,  $\langle \psi_j(\boldsymbol{\xi}) \rangle=0$, $j>0$, $\langle \psi_j(\boldsymbol{\xi})\psi_k(\boldsymbol{\xi}) \rangle=\delta_{jk}$ with
\begin{equation*}
 \langle \psi_j(\boldsymbol{\xi})\psi_k(\boldsymbol{\xi}) \rangle = \int_{\boldsymbol{\xi}\in \Pi} \psi_j(\boldsymbol{\xi})\psi_k(\boldsymbol{\xi})\rho(\boldsymbol{\xi})d\boldsymbol{\xi}.
\end{equation*}
Here $\Pi$ and $\rho$ are the support and probability density of $\xi$, respectively.
By the Cameron-Martin Theorem, the series  converges in the Hilbert space $ L^2(\Omega, \mathcal{F}, \mathcal{P})$, see e.g. \cite{MR2855645}. Thus we truncate (\ref{eqn:infinitesum}),
\begin{equation}
 u(\textbf{x},\omega) = \sum_{j=0}^{P} u_j(\textbf{x}) \psi_j(\boldsymbol{\xi}(\omega)).
\label{eqn:chaostruncated}
\end{equation}
$P$ is determined by the expression
\begin{equation*}
 P = \frac{(N+Q)!}{N!Q!},
\end{equation*}
where $Q$ is the highest degree of the orthogonal polynomial used to represent $u$, and $N$ is used for the approximation of $\alpha$ in (\ref{eqn:kletruncated}).

 \subsection{Stochastic Galerkin approach}

Substituting    the truncated KLE (\ref{eqn:kletruncated}) and truncated generalized solution (\ref{eqn:chaostruncated}) into the model problem (\ref{eqn:stocmain}) gives 
\begin{equation}
-\sum_{j=0}^{P} \nabla .\left(  \left(\bar{\alpha}(\textbf{x}) + \sigma_\alpha \sum_{i=1}^{N}\sqrt{\lambda_i}\varphi_i(\textbf{x})\xi_i(\omega)  \right)\nabla u_j(\textbf{x})\psi_j  \right) - c^2\sum_{j=0}^{P} \left(\bar{\alpha}(\textbf{x}) + \sigma_\alpha \sum_{i=1}^{N}\sqrt{\lambda_i}\varphi_i(\textbf{x})\xi_i(\omega)  \right)u_j(\textbf{x})\psi_j = f(\textbf{x}).
\label{eqn:discrete1}
\end{equation}
\normalsize
We then project (\ref{eqn:discrete1})   onto the space  spanned by the $P+1$ polynomial chaos basis functions to get
\begin{equation}
 -\sum_{j=0}^{P}\nabla .\left( a_{ji}\nabla u_j(\textbf{x})     \right) - c^2 \sum_{j=0}^{P} a_{ji} u_j(\textbf{x})  = \langle \psi_i  \rangle f(\textbf{x}), \quad \text{for}\ \ i=0,...,P,
 \label{eqn:discrete2}
\end{equation} 
\normalsize
where
\begin{equation}
     a_{ji} = \bar{\alpha}(\textbf{x})\delta_{ji} + \sigma_{\alpha} \sum_{k=1}^{N}  \sqrt{\lambda_k}\varphi_k(\textbf{x}) \langle \xi_k \psi_j \psi_i  \rangle.
      \label{eqn:discrete3}
\end{equation}
Note that  (\ref{eqn:discrete2}) consists of $P+1$ deterministic  Helmholtz equation.

\subsection{Basic notations}
Spacial discretization of (\ref{eqn:discrete2}) leads to a linear system that can be represented using Kronecker product and vec$(.)$ operator. Let $G = [g_1,...,g_m] \in \mathbb{R}^{n\times m}$ and $K \in \mathbb{R}^{p\times q}$. Then
\begin{equation*}
G \otimes K = 
\begin{pmatrix}
g_{11}K & g_{12}K & \cdots & g_{1m}K \\
g_{21}K & g_{22}K & \cdots & g_{2m}K \\
\vdots  & \vdots  & \ddots & \vdots  \\
g_{n1}K & g_{n2}K & \cdots & g_{nm}K 
\end{pmatrix}
\in \mathbb{R}^{np\times mq} \quad \text{and} \quad  
\text{vec}(G) = 
\begin{pmatrix}
g_{1} \\
g_{2} \\
\vdots \\
g_{m}  
\end{pmatrix}
\in \mathbb{R}^{nm\times 1}.
\end{equation*} 
vec$(.)$ can be considered as a vector isomophism $\text{vec}(.)\ : \ \mathbb{R}^{n \times m} \rightarrow \mathbb{R}^{nm}$    and its inverse is denoted by vec$(.)^{-1}\ : \ \mathbb{R}^{nm} \rightarrow \mathbb{R}^{n \times m}$.

In MATLAB notation: $\text{vec}(G)=\text{reshape}(G,[n*m,1])$ and
$G=\text{vec}^{-1}(\text{vec}(G)) = \text{reshape}(\text{vec}(G),[n,m])$.
The following properties  hold 
\begin{equation*}
 \text{vec}(AXB) = (B^T\otimes A)\text{vec}(B)
\end{equation*}
and 
\begin{equation}
 (A\otimes B)(C\otimes D) = AC \otimes BD.
 \label{eqn:kroncarp}
\end{equation}

\subsection{Spatial discretization}

Discretization of (\ref{eqn:discrete1}) with a standard method leads to system of  ill-conditioned indefinite matrices  for large $ch$ \cite{KAYA2022171}. Furthermore, standard discretization schemes such as the standard finite difference or finite element methods,  suffer from the pollution effectcfor large $c$ \cite{MR1480387}. The simple way to eliminate the pollution error is to choose $c^2h<1$. However, this leads to intractable matrix sizes, particularly for the model problem (\ref{eqn:stocmain}).  There are different methods in the literature which are designed to reduce the pollution error such as high-order finite element methods \cite{doi:10.1137/S0036142903423460, MR1445739, MR2812565, MR4092293}, multi-scale methods \cite{MR4079241, MR3614010}, Trefftz methods \cite{MR1618464, MR1870426,MR1692914},     stabilized methods \cite{MR1475348, https://doi.org/10.1002/nme.1620380303, adem} and others \cite{MR3809250, MR38092508888}.   We use the stabilized finite element method proposed in  \cite{adem} to  reduce the pollution error.  Using a     nonstandard method that is successful in reducing the pollution effect, reduces the size of the resulting system matrix for the model problem (\ref{eqn:stocmain}), significantly. For more details about the construction of the system matrix, we refer to \cite{adem, MR3308418}. 

Although a nonstandard method may reduce the pollution error, the system matrix generally remains indefinite for large wave numbers. It is well known that   standard iterative solvers are ineffective in obtaining  solutions of the  discrete Helmholtz equation \cite{MR3050918}. We will show that this will be also the case for the discrete model problem (\ref{eqn:stocmain}) for large wave numbers.

We assume that each of the deterministic coefficients $u_j$ , $j =0,1,. . . , P$, in (\ref{eqn:discrete2}) is discretized on the same mesh and with equal number of elements. More precisely, each $u_i$ is approximated  as a linear
combination of the form 
\begin{equation*}
 u(\textbf{x}) \approx  \sum_{i=1}^{J}u_{ji}s_i(\textbf{x}), \quad j=0,...,P.
\end{equation*}
with $J$  basis functions $s_i(\textbf{x})$ such that $s_i(\textbf{x}) = s_L^i + s_B^i$ where $s_L^i$ are the linear finite element basis functions and $ s_B^i$ are the bubble functions defined in \cite{adem} for the stabilization of the method.

After applying the spacial discretization to (\ref{eqn:discrete2}) with (\ref{eqn:discrete3}), we get  fully discretized form of (\ref{eqn:stocmain}) which is equivalent to the following linear system  
\begin{equation}
 \mathcal{A}\textbf{u}=\textbf{b},
\label{eqn:mainsystem}
\end{equation}
where 
\begin{equation}
 \mathcal{A} = \sum_{i=0}^{N} G_i \otimes K_i
\label{eqn:mainmatrix}
\end{equation}
and 
\begin{equation*}
 \textbf{b} = \textbf{g}_0 \otimes \textbf{f}_0 - \textbf{g}_0 \otimes K_{0,B}^T \textbf{w}_B(x).
\end{equation*}
The stochastics matrices  $G_i \in \mathbb{R}^{(P+1)\times(P+1)}$ are given by
\begin{equation*}
 G_0(j,k) = \langle \psi_j(\xi) \psi_k(\xi)   \rangle, \quad G_i(j,k) = \langle \xi_i \psi_j(\xi) \psi_k(\xi), \quad \text{for}\quad i=1,...,N. 
\end{equation*}
The stiffness matrices $K_i \in \mathbb{R}^{J\times J}$, for $i=1,...,N$ are given by
\begin{equation*}
 K_0(j,k) = \int_D \bar{\alpha}(\textbf{x})\nabla s_j(\textbf{x}) \nabla s_k(\textbf{x}) d\textbf{x} - c^2 \int_D \bar{\alpha}(\textbf{x}) s_j(\textbf{x}) s_k(\textbf{x}) d\textbf{x},
\end{equation*}
\begin{equation*}
 K_i(j,k) =\sigma_{\alpha}\sqrt{\lambda_i} \int_D \varphi_i(\textbf{x})\nabla s_j(\textbf{x}) \nabla s_k(\textbf{x}) d\textbf{x} - c^2 \sigma_{\alpha}\sqrt{\lambda_i} \int_D \varphi_i(\textbf{x}) s_j(\textbf{x}) s_k(\textbf{x}) d\textbf{x}
\end{equation*}
 and the vectors  $\textbf{g}_0$ and $\textbf{f}_0$ are defined by 
\begin{equation*}
 \textbf{g}_0(j) = \langle \psi_j(\xi) \rangle, \qquad \textbf{f}_0(j) = \int_{D} f(\textbf{x}) s_j(\textbf{x}) d\textbf{x}. 
\end{equation*}
$K_{0,B} \in \mathbb{R}^{j_b \times J}$ is  a finite element  matrix that accounts for
the coupling between interior and boundary degrees of freedom and $\textbf{w}_B \in \mathbb{R}^{j_b}$ contains the boundary data.

\section{Existence of low-rank solutions for indefinite linear systems } \label{sec:existence}
Our  aim in this article is to apply a low-rank algorithm to the system (\ref{eqn:mainsystem}) which has kronecker structure. To this end, we discuss existence  of a low-rank solution. Existence theory for general positive definite matrices  is considered in  \cite{MR2081823} and for discrete stochastic problems in \cite{MR3376791}. We first establish, under certain assumptions existence theory for a low-rank inverse for a general indefinite matrix  following the theory in \cite{MR2081823}.

\begin{lemma} \label{lemma:inverse} Let $G \in \mathbb{C}^{n \times n}$. If the spectrum of $G$ is  contained in the upper half complex  plane, i.e., 
\begin{equation*}
 \sigma(G) \subset \mathbb{C}_{+_i} := \{ x+iy \in \mathbb{C}\ | \ y>0  \},
\end{equation*}
then the inverse of $G$ is
\begin{equation}
 G^{-1} = -i\int_{0}^{\infty} e^{iGt}dt.
\end{equation}
\end{lemma}

\begin{proof}  
\begin{equation*}
 G\left(-i\int_{0}^{\infty} e^{iGt}dt  \right) = -\frac{i}{i}  \int_{0}^{\infty} \frac{\partial}{\partial t}\left( e^{iGt} \right) dt = e^{0G} = I. 
\end{equation*}
\end{proof}

\begin{lemma} Let $A$  possess  the tensor structure
\begin{equation}
 A = \sum_{j=1}^d \hat{A_j}, \qquad \hat{A_j}= \underbrace{I\otimes ...\otimes I}_{j-1 \ \text{times}}\otimes A_j \otimes \underbrace{I\otimes ...\otimes I}_{d-j \ \text{times}}, \qquad A_i \in \mathbb{C}^{n\times n}. 
 \label{eqn:matrixtensor}
\end{equation}
If the  the spectrum of $A$  is contained in the upper half complex plane, then the
inverse of $A$ is
\begin{equation}
 A^{-1} = -i\int_{0}^{\infty} \sbigotimes_{j=1}^{d}  e^{iA_jt}dt.
\end{equation}
\end{lemma}

\begin{proof}
\proof Making use of  Lemma \ref{lemma:inverse} and  following the steps in (\cite{MR2081823}, Lemma 2) yields the result.
\end{proof}

One can also    consider   $iAx=ib$ for which the eigenvalues of $A$ are rotated in the complex plane. If the eigenvalues of $A$ are in the upper half complex plane, then the eigenvalues of $iA$ are in in the left half complex plane so that we can apply the  theory proposed in  \cite{MR2081823} for $iA$.

\begin{lemma} \label{lem:lem3} Let  $A$ be a matrix of the tensor structure (\ref{eqn:matrixtensor}) with spectrum $\sigma(A)$ contained in the strip  $\Lambda:= [-\lambda, \lambda]\otimes i[\mu_{\text{min}},\mu_{\text{max}}] \subseteq \mathbb{C}_{+_i}$. Let $b$ possesses the  tensor vector structure
\begin{equation*}
 b = \sbigotimes_{i=1}^{d}b_i, \qquad b_i\in \mathbb{C}^n, \quad b_{(j_1,...j_d)} = \prod_{i=1}^d (b_j)j_i \quad \text{for} \ j=\{1,...d\}^d
\end{equation*}
and $k\in \mathbb{N}$. Then the solution
$x$ to $Ax= b$ can be approximated by
\begin{equation}
\tilde{x} : = -\sum_{j=-k}^{k} \frac{2w_j}{\mu_{\text{min}}}  \sbigotimes_{i=1}^{d}  e^{\frac{2t_j}{\mu_{\text{min}}}A_i}b_i
\end{equation}
with approximation error (in the Euclidean norm)
\begin{equation}
 ||x-\tilde{x}||_{2} = -\sum_{j=-k}^{k} \frac{C_{st}}{\pi \mu_{\text{min}}} e^{\frac{2\lambda \mu_{\text{min}}^{-1} +1 }{\pi} -\pi \sqrt{2 k}} \oint_{\Gamma} \left|\left| (\mu I- 2A/\mu_{\text{min}})^{-1} \right|\right|_{2}  d_{\Gamma}\mu \ ||b||_2,
 \label{eqn:lowrankerr}
\end{equation}
where 
\begin{equation*}
 h_{st} := \frac{\pi}{\sqrt{k}}, 
\end{equation*}
\begin{equation}
  t_j:= \log\left(e^{jh_{st}} + \sqrt{1+e^{2jh_{st}}}  \right),
\label{eqn:tj}
\end{equation}
\begin{equation}
 w_j:= \frac{h_{st}}{\sqrt{1+e^{-2jh_{st}}}},
 \label{eqn:wj}
\end{equation}
and $C_{st}$ is a constant independent of $k$. 
\end{lemma}

We also need the Sherman-Morrison-Woodbury formula \cite{MR35118,MR40068}.

\begin{lemma}\label{lem:lem4} Let $X \in \mathbb{C}^{n\times n}$  be nonsingular and let $Y, Z \in \mathbb{C}^{n\times m} $, with $m \leq n$. Then $X+YZ^T$ is invertible if and only if $I_m + Z^TX^{-1}Y$ is invertible  and
\begin{equation}
 (X+YZ^T)^{-1} = X^{-1} - X^{-1}Y(I_m + Z^TX^{-1}Y )^{-1} Z^T X^{-1}.
\label{eqn:sherman}
\end{equation} 
\end{lemma}

Using Lemmas \ref{lem:lem3} and \ref{lem:lem4} we state the main result. To this end, we split the system matrix (\ref{eqn:mainmatrix}) as follows
\begin{equation}
 \mathcal{A} = L + \sum_{i=1}^{N}G_i \otimes K_i,
 \label{eqn:matrixl}
\end{equation}
where $L=G_0 \otimes K_0$. If $K_0$ has eigenvalues in the upper half complex plane, then $L$ has eigenvalues in the upper half complex plane as well when $G_0$ is the identity matrix.   Note that  symmetry for the matrices $K_i$, $i=0,...,N$, is slightly spoiled as a nonstandard finite element method is used. Let the stochastic matrices $G_i$, $i=1,...,N$ be decomposed as
\begin{equation}
 G_i:= U_iV_i^T, \quad U_i, V_i \in \mathbb{C}^{(P+1)\times r_i}, \ i=1,...,N, \quad r_i << P+1.
 \label{eqn:matrixg}
\end{equation}
Furthermore, let the stiffness matrices $K_i$, $i=1,...,N$ be decomposed as 
\begin{equation}
 K_i:= \tilde{U}_i \tilde{V}_i^T, \quad \tilde{U}_i, \tilde{V}_i \in \mathbb{C}^{J\times J}, \ i=1,...,N.
 \label{eqn:matrixklov}
\end{equation}
The following result holds. 
 
\begin{thm}
Let $\mathcal{A}$ be a matrix of Kronecker product structure as in (\ref{eqn:mainmatrix}). Assume that the spectrum of $L$ in (\ref{eqn:matrixl}) is contained in the strip  $\Lambda:= [-\lambda, \lambda]\otimes i[\mu_{\text{min}},\mu_{\text{max}}] \subseteq \mathbb{C}_{+_i}$ and let $\Gamma$ be the boundary of $[1,2\mu_{\text{max}}/\mu_{\text{min}}+1]$. Let $G_i$ and $K_i$, $i=1,...,N$, have the low-rank representation as in (\ref{eqn:matrixg}) and (\ref{eqn:matrixklov}), respectively. Suppose also that $U=[U_1\otimes \tilde{U}_1,...,U_N\otimes \tilde{U}_N]$ and   
$V=[V_1\otimes \tilde{V}_1,...,V_N\otimes \tilde{V}_N]$. Let the tensor rank of $\mathbf{b} \leq \ell$, where $\ell \ll J(P+1)$. Then, for $k\in \mathbb{N}$, the solution  $\mathbf{u}$  of (\ref{eqn:mainsystem})  can be approximated by
a vector $\tilde{\mathbf{u}}$  of the form
\begin{equation}
\tilde{\mathbf{u}} = -\sum_{j=-k}^{k} \frac{2w_j}{\mu_{\text{min}}} \left( e^{G_0} \otimes e^{-\frac{2t_j}{\mu_{\text{min}}}\tilde{K}_0}  \right) (\mathbf{b} -U\mathbf{y})
\label{eqn:approximateu}
\end{equation}
where the vector $\mathbf{y}\in \mathbb{C}^{J.r}$ is the solution of
\begin{equation}
 (I_{J.r} + V^T L^{-1}U)\mathbf{y} = V^TL^{-1}\mathbf{b} 
\end{equation}
and $t_j$ , $w_j$ are the quadrature weights and points as given by (\ref{eqn:tj}) and (\ref{eqn:wj}). The
corresponding approximation error is given by
\begin{equation}
 ||\mathbf{u}-\tilde{\mathbf{u}}||_2 \leq \frac{C_{st}}{\pi \mu_{\text{min}} } e^{(\pi^{-1}- \pi \sqrt{k}  )} |\Gamma|||\mathbf{b}- U\mathbf{y}||_2.
 \label{eqn:errorlow}
\end{equation}
\end{thm}
\begin{proof}  Using  (\ref{eqn:kroncarp}), (\ref{eqn:matrixg}) and (\ref{eqn:matrixklov}), we have  the low-rank representation
\begin{equation}
 \sum_{i=1}^{N}G_i \otimes K_i = \sum_{i=1}^N (U_i V_i^T)(\tilde{U}_i \tilde{V}_i^T) = \sum_{i=1}^N (U_i \otimes \tilde{U_i})(V_i^T \otimes \tilde{V_i}^T) = UV^T.
\label{eqn:prooflow} 
\end{equation}
Hence, from Lemma \ref{lem:lem4}, (\ref{eqn:matrixl}) and (\ref{eqn:prooflow}), we have that 
\begin{equation}
 \mathcal{A}^{-1} =(L+UV^T)^{-1} = L^{-1} - L^{-1}U(I_{J.r}+V^TL^{-1}U)^{-1}V^TL^{-1}, 
\end{equation} 
so that 
\begin{equation}
 \mathbf{u} = \mathcal{A}^{-1} \mathbf{b} =  L^{-1} (\mathbf{b} - U\mathbf{y}),
 \label{eqn:proof51}
\end{equation}
where $\mathbf{y} = (I_{J.r}+V^TL^{-1}U)^{-1}V^TL^{-1} \mathbf{b}$. By assumption, the matrix $L$ has eigenvalues in the upper half   complex plane.  Thus, using the fact that 
\begin{equation}
 e^{-\beta L} = e^{-\beta (G_0\otimes K_0)} = e^{ G_0\otimes (-\beta K_0)} = e^{ G_0}\otimes e^{-\beta K_0}, 
\end{equation}
where $\beta:= 2t_j/\mu_{\text{min}}$, together with (\ref{eqn:proof51}) and Lemma \ref{lem:lem3}, immediately yields (\ref{eqn:approximateu}) and (\ref{eqn:errorlow}). 
\end{proof}

Note that  $\mu_{\text{min}}^{-1}$ appears as a factor in (\ref{eqn:approximateu}), (\ref{eqn:errorlow}) which means that one needs to use more quadrature points as $\mu_{\text{min}}$ get closer to zero to keep the error fixed. This corresponds to using more singular values in a low-rank algorithm. Hence, we expect to see more singular values to be used in  our low-rank algorithm for the stochastic Helmholtz problem as the wave number $c$ increases.

\section{Computation of low-rank approximations} \label{sec:lowrank}

The  matrix $\mathcal{A}$ in (\ref{eqn:mainsystem}) is nonsymmetric and indefinite. Moreover, it is generally ill-conditioned with respect to the stochastic and spatial discretization parameters, e.g. the finite element mesh size, wave number, the length $N$ of the
random vector $\xi$, or the total degree of the multivarite stochastic basis polynomials $\{\psi_i \}$, \cite{MR2491431}.  A natural iterative solver for the system is a preconditioned BICG method \cite{MR0461857} as the matrix is complex-valued. GMRES requires more storage than BICG, hence GMRES is not considered here.  The   CGS and the BICGSTAB methods are other candidates but our numerical experiments suggest that the BICG method is  better in terms of number of iterations. We do not provide numerical experiments for CGS and BICGSTAB in this work.  Since the size of   $\mathcal{A}$ is very large, we  use low-rank preconditioned BICG which does not require construction of $\mathcal{A}$ explicitly. Another important point is the choice of the preconditioner and  its implementation in a low-rank algorithm. In what follows, we examine these points.

\normalfont
\subsection{Preconditioning} 

In a standard Krylov subspace method, matrix vector multiplication, inner products of vectors and  vector addition are very cheap in comparison with the application of a preconditioner. In a low-rank algorithm, obtaining the singular values with, for example, Matlab's \textit{svd}, is an extra operation, which might not be so cheap. One can accelerate a low-rank algorithm, for example, by computing the SVD  in a more efficient way, which is an ongoing work. Moreover, depending on the problem, matrix vector multiplication might not be cheap enough if it is not defined appropriately.   On the other hand, it is possible to further accelerate a low-rank algorithm by applying the preconditioner in a more efficient way. In this regard, the structure of the preconditioner matrix $\mathbf{M}$ is very important. In this article, we will apply the precondioner in the low-rank BICG method in a suitable way. While we  are preserving the low-rank structure of the approximate solution, we  apply the preconditioner  more efficiently, so that we  obtain a  gain in CPU time.

Although different preconditioners are available for the discretizations of the stochastic PDEs , we use  the mean-based preconditioner, \cite{Ghanem1996NumericalSO}, for our purpose, which is  given by
\begin{equation}
 \mathbf{M}_0 := G_0 \otimes K_0.
\end{equation}
Note that $G_0$ is a diagonal matrix as the stochastic basis functions $\{\psi_i \}$ are orthogonal. $K_0$ is indefinite for large wave numbers and hence, $\mathbf{M}_0$ is also indefinite and $\mathbf{M}_0^{-1} = G_0^{-1}\otimes K_0^{-1}$, where $G_0^{-1}(i,i)=1/G_0(i,i)$.  While Matlab uses backslash for  $\mathbf{M}_0$ to implement the preconditioner for Krylov methods, the diagonal block structure of $\mathbf{M}_0$ allows for  implementation of the preconditioner by solving $K_0x_k=r_k$ for $k=1,...,Q$, as done in \cite{MR3376791}, where an algebraic multigrid method is used as a preconditioner for $K_0$ when solving a stochastic diffusion problem. However, this approach is not an advantage of a low-rank method, because one has to turn back to full-rank format for the solution in either matrix form or vector form.   As we will see later, the most time consuming step in a plain Krylov method is the implementation of the preconditioner. In order to do a fair comparison between a low-rank method and a full-rank method, we should implement the preconditioner in such a way that it must be appropriate for a low-rank method. To this end, we propose a general approach for the discretizations of the stochastic PDEs using low-rank techniques to implement the preconditioner such that it accelerates  a low-rank method.
Note that the vectors in a low-rank algorithm are in the form of product of two matrices such as $\text{vec}^{-1}(x)=X_uX_v^T=[x_1,...,x_Q]$ and $\text{vec}^{-1}(r)=R_uR_v^T=[r_1,...,r_Q]$.
To implement the preconditioner, i.e., to solve $\mathbf{M}_0x=r$, we calculate the inverse of $K_0$ once, then we  set
\begin{equation}
  X_u=K_0^{-1}R_u \quad  \text{and} \quad  X_v=R_v.
 \label{eqn:peocondimplement}
\end{equation}

The inverse of $K_0$ is a full matrix and this requires more memory. However, compared to the nonzero entries of $\mathcal{A}$, there is still a big advantage in terms of memory usage. When the rank $k$ is small, the product $K_0^{-1}R_u$ is very cheap. We will show in numerical tests that even for large rank $k$, it is still faster than standard  approaches. One of the  advantages is that it does not require another preconditioner for $K_0$, which may be difficult to find.   
Moreover, the implementation of the preconditioner is very simple  this way.

\subsection{Preconditioned low-rank iterative solvers}

Having discussed  the  preconditioner, we now proceed in this section to show the implementation of the preconditioned low-rank BICG (P-LRBICG) algorithm.   
We present P-LRBICG in Algorithm \ref{alg:algorithm1}.

\begin{algorithm}[h]
\SetAlgoLined
\KwData{Matrix functions $\mathcal{A}, \mathbf{M}: \mathbb{R}^{n\times m} \rightarrow \mathbb{R}^{n\times m} $, right hand side $B \in \mathbb{R}^{n\times m}$, $\tilde{R}_0 \in \mathbb{R}^{n\times m}$  in low-rank format (e.g.,$\tilde{R}_0 = B$). Truncation operator $\mathcal{T}$ w.r.t. relative accuracy $\epsilon_\text{rel}$.}
\KwResult{Matrix $X \in \mathbb{R}^{n\times m} $ satisfying $||\mathcal{A}X- B||_F \leq tol$.}
 Initialization; $X_0=0$, $R_0=B$, , $\tilde{R}_0=R_0$, $k=0$ \;
 \While{$||R_k||_F > tol$}{
  $Z_k = \mathbf{M}^{-1}R_k$  \;
  $\tilde{Z}_k =(\mathbf{M}^H)^{-1}\tilde{R}_k$ \;
  $\rho_1=\rho$ \;
  $\rho = \langle \tilde{R}_k, Z_k   \rangle$  \;
  \textbf{If} $k=0$ \textbf{then} \; $P_k=Z_k$, $\tilde{P}_k=\tilde{Z}_k$ \; \textbf{else} \; $\beta=\rho_0/\rho_1$\; 
   $P_k = Z_k+\beta P_k$    \qquad \quad    $P_k\leftarrow \mathcal{T}(P_k)$ \;
   $\tilde{P}_k = \tilde{Z}_k + \bar{\beta}\tilde{P}_k$  \qquad \quad  $\tilde{P}_k\leftarrow \mathcal{T}(\tilde{P}_k)$ \;
   \textbf{end if} \;
  $Q_k = \mathcal{A}P_k$  \;
  $\tilde{Q}_k = \mathcal{A}^H\tilde{P}_k$  \;  
  $ptq = \langle \tilde{P}_k,Q_k  \rangle $  \;
  $\alpha = \rho/ptq$  \;
  $X_k = X_k + \alpha P_k$   \qquad \quad    $X_k\leftarrow \mathcal{T}(X_k)$  \;
  $R_k = R_k-\alpha Q_k$   \qquad \quad    $R_k\leftarrow \mathcal{T}(R_k)$  \;
  $\tilde{R}_k = \tilde{R}_k - \bar{\alpha} \tilde{Q}_k $  \qquad \quad  $\tilde{R}_k\leftarrow \mathcal{T}(\tilde{R}_k)$  \;
    $k = k+1  $
 }
 \caption{Preconditioned low-rank BiCG method (P-LRBICG)}
 \label{alg:algorithm1}
\end{algorithm}

We clarify the operations in the P-LRBICG algorithm. We have already discussed the implementation of the preconditioner (see Eqn: \ref{eqn:peocondimplement} ). Matrix vector multiplication is defined in \cite{ MR3376791,ciloglu2020stochastic}.
\begin{equation*}
 \mathcal{A}\text{vec}(X) = \sum_{i=1}^N \sum_{j=1}^k \left(G_i v_i  \right) \otimes \left(K_i v_i  \right) \in \mathbb{R}^{J.(P+1)\times 1}
\end{equation*}
where 
\begin{equation*}
 X = UV^T, \quad U\in \mathbb{R}^{J\times k}, \quad V\in \mathbb{R}^{(P+1)\times k}, \quad k \ll J,P,
\end{equation*}
\begin{equation*}
 U = [u_1,...,u_k ], \quad V = [v_1,...,v_k].
\end{equation*}
Although this approach is effective for small rank, it gets more inefficient for larger  rank as it uses two inner loops. Here, we do the matrix vector multiplication in a different and more efficient way.

Using the property of Kronecker product $(A^T\otimes B)\text{vec}(X) =\text{vec}(AXB)$, we perform the matrix vector multiplication.
\begin{align}
 \mathcal{A}\text{vec}(X) = \left(\sum_{i=0}^N (G_i \otimes K_i)\right)\text{vec}(X) = \sum_{i=0}^N (G_i \otimes K_i)\text{vec}(X) = \text{vec}\left(\sum_{i=0}^N  K_i X G_i \right)
 \label{eqn:matrixvecmultip}
\end{align}
where $X=U_{old}V_{old}^T$. We then set $U_{new}=\sum_{i=0}^N  K_i X G_i$ and $V_{new}=I$.

Vector addition in this low-rank algorithm  amounts to writing matrices of low-rank side by side \cite{MR3296628,MR3759421}. The truncation is done by the \textit{svds} function of Matlab.

Note that this implementation of LRBICG naturally leads to an inexact BICG method. For more details on inexact Krylov solvers, we refer to the literature \cite{MR1974182,MR2327109}.

\section{Numerical experiments} \label{sec:numerical}

In this section, we report some numerical results to show the performance of the low-rank approach presented in this paper applied to (\ref{eqn:stocmain}). We choose academic problems rather than real life applications in order to be consistent with the theory proposed in Section 3 and to show the performance of the each operations in Algorithm 1. We expect to see similar results for more advanced applications. We choose $f=0$, $g(x)=0.1$ and the spatial domain $D$ as illustraed in Figure \ref{fig:2ddomain}. Uniform triangular elements are used for the discretization of the domain using  $N_p$ points both in $x-$ and $y-$directions. 
\begin{figure}[h]
\center
 \includegraphics[width=6.5cm]{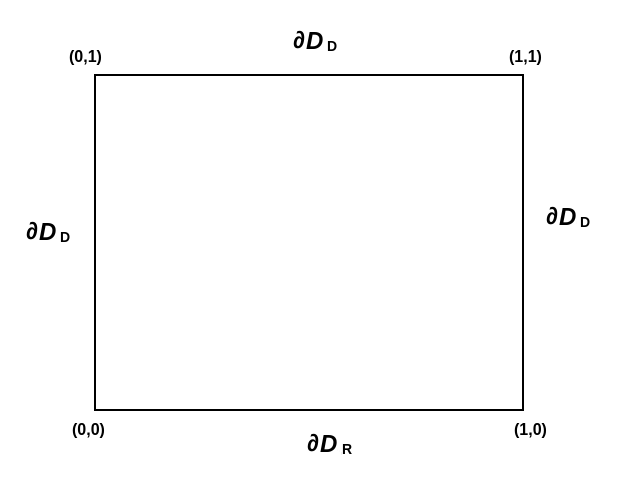} 
\caption{The spatial domain $D$. }  
\label{fig:2ddomain} 
\end{figure} 
The eigenpairs ($v_{k}, \phi_{k} $) of the KL-expansion of $\alpha$ are obtained by reordering the $v_{i,j}$ and $\phi_{i,j}$ which are given explicitly in \cite{MR3308418}:
\begin{equation*}
 v_{i,j} = \frac{1}{4} e^{-\pi(i^2+j^2)l^2}, \qquad \phi_{i,j} = 2\cos(i\pi x) \cos(j\pi y) \qquad i,j\geq 1.
\end{equation*}
Moreover, we set $\bar{\alpha}=5$.  In the numerical experiment, we set $l=1$ and investigate the behavior of the P-LRBICG for
different values of the discretization parameters $P(N, Q, N_{p},c)$.
Moreover, we choose $\xi=[\xi_1,...,\xi_N  ]$, $\xi_j \sim \mathbf{U}[-\sqrt{3},\sqrt{3}]$, and and $N$-dimensional Legendre polynomials with support in $[-\sqrt{3},\sqrt{3}]^N$ (uniform distribution). We choose $\psi_1=1$ so that $G_0$ is the identity matrix.  We do not consider  other polynomials, because the main difficulty arises from the finite element matrices $K_i$. We expect to see similar results when other polynomials are used. 
   
 A Linux machine with 16 GB RAM, Intel core i5 cpu and 
 MATLAB$^{\circledR}$  9.8 (R2020a) was used to perform the numerical experiments.
The stopping criterion for all numerical experiments for both P-LRBICG and BICG was $10^{-4}$ and the tolerance for the truncation operator in all cases was $10^{-6}$. We calculate all  singular values, and then cut off certain singular values according to the tolerance. We note that it is possible to further decrease the CPU time for the truncation operator by using different approaches \cite{palitta}, particularly for the positive definite case.

 In the numerical experiments, we report the  CPU time for the  P-LRBICG method and for the plain BICG method, the number of iterations for convergence, the matrix sizes, the
 CPU time for the operations in the P-LRBICG method and the CPU time of the preconditioner in the plain BICG method, the average  number of singular values used per iteration  and the sum of the nonzero entries of the matrices (nnz). The nonzero entries of the inverse matrix $K_0$ are included for P-LRBICG method.

\subsection{Test 1 (positive definite case) } We first test preconditioned-LRBICG  when $c=0$  for which the system matrix $A$ in (\ref{eqn:mainmatrix}) is 
positive definite. Altough the conjugate gradient method is more efficient in this case, our aim is to asses the performance of the preconditioner when 
it is implemented as discussed before. Matlab uses backslash to implement the preconditioner in a plain Krylov method. We compare the 
CPU time of preconditioner for P-LRBICG  and plain P-BICG in Table \ref{table:1} for varying $N_p$. 

Observe that the CPU time of the preconditioner for the low-rank approach is much smaller than for the standard approach. Almost all of the total CPU time of the plain BICG approach stems from the preconditioner. The most expensive operation in the low-rank approach is application of the truncation operator. The trace operator and matrix vector multiplication are quite cheap in comparison with the truncation operator and the precondioner. The total CPU time for the low-rank approach is much smaller than for the standard approach.  We note that the gain in terms of CPU time using low-rank approach increases  as the size of the system matrix increases but the gain in terms of storage decreases. However, there is still a big advantage in terms of storage for the low rank solver over full rank solver.

\begin{table}[h!]
\begin{center}
\begin{tabular}{ |p{3.1cm}|p{1.7cm}|p{1.5cm}| p{3.5cm} | p{3.5cm}| } 
\hline
    & P-LRBICG \ \ \ \ \ \ \  Iterations & P-BICG \ \ \ \ \ \ \ \ \ Iterations & P-LRBICG CPU time \& average number of singular values  & P-BICG CPU time \\ 
\hline
$P(6,4,11,0)$ \ \ \ \ \ \  (size$(\mathcal{A})=18480$) & 3 & 3 & \textbf{Total CPU time = 0.67}   \ \ \ \ \ \ \ \ Preconditioner = 0.1\% \ Truncation = 47.7\% \ \ \ \ \ \ \ \ \ Trace operator = 16.4\% \  Mrx. vec. mltp. = 22.4\% \ \ \ \ \ \ \ \ \ \ \ \ \  Avr. \# of singular val. = 18.7,  \  nnz = 10,928  & \textbf{Total CPU time = 0.53} \ \ Preconditioner = 79.2\%  \  nnz = 776,832\\ 
\hline
 $P(6,4,21,0)$  \ \ \ \ \ \  (size$(\mathcal{A})=79380$) & 3 & 3 &  \textbf{Total CPU time = 2.57}   \ \ Preconditioner = 6.2\% \ Truncation = 61.4\% \ \ \ \ \ \ \ \ \ Trace operator = 9.3\% \ \ Mrx. vec. mltp.  = 15.7\% \ Avr. \# singular val. = 18.7, \ nnz = 153,868   & \textbf{Total CPU time = 5.9} \ \ Preconditioner = 95.4\%  \  nnz = 3,561,432  \\ 
 \hline
 $P(6,4,41,0)$  \ \ \ \ \ \  (size$(\mathcal{A})=327180$) & 3 &  3  & \textbf{Total CPU time = 10.5}   \ \ Preconditioner = 21.3\% \ Truncation = 47.3\% \ \ \ \ \ \ \ \ \ Trace operator = 10.6\% \ \ Mrx. vec. mltp.  = 16.3\% \ Avr. \# singular val. = 18.0, \  nnz = 2,470,748   & \textbf{Total CPU time = 28.3} \ \ Preconditioner = 96.9\% \  nnz = 15,128,232 \\ 
 \hline
 $P(6,4,61,0)$  \ \ \ \ \ \  (size$(\mathcal{A})=742980$) & 3 &  3   & \textbf{Total CPU time = 30.7}   \ \ Preconditioner = 37.6\% \ Truncation = 39.5\% \ \ \ \ \ \ \ \ \ Trace operator = 7.8\% \ \ Mrx. vec. mltp.  = 13.0\% \ Avr. \# singular val. = 17.3,  nnz = 12,615,628   & \textbf{Total CPU time = 72.0} \ \ Preconditioner = 96.7\% \  nnz = 34,691,832  \\ 
\hline
 $P(6,4,81,0)$  \ \ \ \ \ \  (size$(\mathcal{A})=1326780$)  & 3 &  3   & \textbf{Total CPU time = 73.5}   \ \ \ Preconditioner = 52.7\% \ Truncation = 29.1\% \ \ \ \ \ \ \ \ \ Trace operator = 6.1\% \ \ Mrx. vec. mltp.  = 10.2\% \ Avr. \# singular val. = 16.3,  nnz = 40,092,508   & \textbf{Total CPU time = 136.9} \ \ Preconditioner = 97.4\% \ nnz = 62,252,232
 \\ 
\hline 
\end{tabular}
\caption{Results of Test 1 showing total CPU times, total iterations, average number of singular values used per iteration and CPU times for different operations   from
preconditioned low-rank BICG (P-LRBICG) solver  compared with
those from plain preconditioned BICG (P-BICG) for selected parameters $P(N, Q, N_{p},c)$. size$(\mathcal{A})$ represents the size of the matrix $\mathcal{A}$.   }
\label{table:1}
\end{center}
\end{table}

 \subsection{Test 2 (indefinite case) }

 We now increase the wave number gradually and observe the  behavior of the P-LRBICG method and compare with the plain BICG method. We first show the eigenvalues of the matrix $L=\mathbf{M}_0$ in Figure \ref{fig:eigenvalues} for selected problem parameters to be consistent  with the existence theory proposed in Section \ref{sec:existence}. Note that  eigenvalues of $L$ are the same as the eigenvalues of $K_0$ as $G_0$ is an identity matrix. Observe that all eigenvalues are in the upper half complex plane which is a requirement for the  existence theory of low-rank solutions, and more eigenvalues lie in the second quadrant as the wave number $c$ increases.

While we report simulations result in Table \ref{table:2} for the problem  $P(6,4,41,c)$ for increasing $c$, we change the stochastic parameters in Table   \ref{table:3} and reports results for the problem  $P(5,3,41,c)$ for increasing $c$. In Table \ref{table:4}, we further increase the values of the stochastic parameters and consider the problem $P(7,5,21,c)$ .  In all cases, more singular values are used per iteration as the wave number $c$ increases. This increases cost of  all operations in a low-rank approach. However, it is still possible to get advantage in total CPU time for large wave numbers. Observe that the implementation of the preconditioner  is still much more efficient  for the low-rank approach than in  the plain BICG method. For the problem $P(7,5,21,c)$, the implementation of the preconditioner is relatively cheap, less than $2\%$ of the total CPU time.  There is also an advantage in terms of storage.    One  point is that we do not see a blow up in CPU time for the matrix vector multiplication due to the new definition in (\ref{eqn:matrixvecmultip}). Otherwise, it could be more expensive than applying the truncation operator. The number of iterations required for the convergence increases as the wave number increases which is consistent with the observation in Figure  \ref{fig:eigenvalues} which shows  that the eigenvalue distribution spreads further for large wave numbers.

\begin{figure}[h!]
\center
 \includegraphics[width=8.0cm]{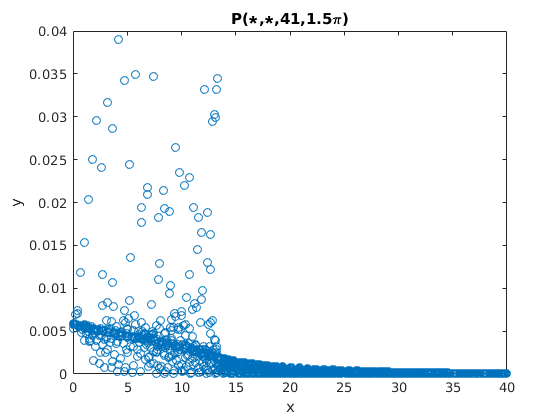}
 \includegraphics[width=8.0cm]{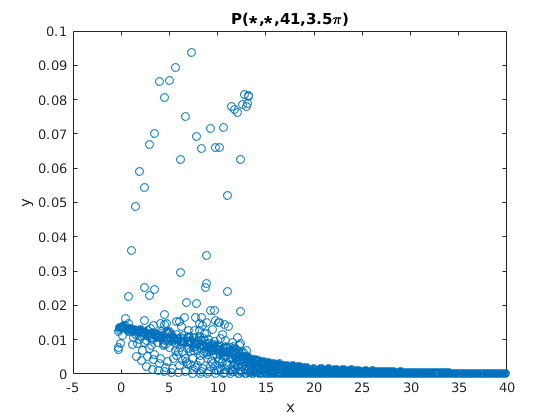}
 \includegraphics[width=8.0cm]{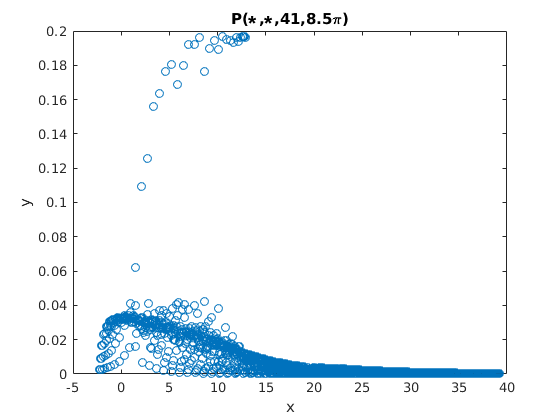}
\caption{Eigenvalues of the matrix $K_0$ for different values of the wave number $c$.} 
\label{fig:eigenvalues}
\end{figure}

\begin{table}[h!] 
\begin{center}
\begin{tabular}{ |p{3.1cm}|p{1.7cm}|p{1.5cm}| p{3.5cm} | p{3.5cm}| } 
\hline
    & P-LRBICG \ \ \ \ \ \ \  Iterations & P-BICG \ \ \ \ \ \ \ \ \ Iterations & P-LRBICG CPU time \& average number of singular values  & P-BICG CPU time \\ 
\hline
$P(6,4,41,1.5\pi)$ \ \ \ \ \ \  (size$(\mathcal{A})=327180$) & 10 & 10 & \textbf{Total CPU time = 39.0}   \ \ Preconditioner = 30\% \ Truncation = 42.5\% \ \ \ \ \ \ \ \ \ Trace operator = 10.5\% \ \ Mrx. vec. mltp.  = 13.5\% \ Avr. \# of singular val. = 36.72, \  nnz = 2,470,748   & \textbf{Total CPU time = 78.9} \ \ Preconditioner = 97.1\% \  nnz = 15,128,232 \\ 
\hline
 $P(6,4,41,3.5\pi)$  \ \ \ \ \ \  (size$(\mathcal{A})=327180$) & 21 & 21 &  \textbf{Total CPU time = 87.8}   \ \ Preconditioner = 31.8\% \ Truncation = 42.5\% \ \ \ \ \ \ \ \ \ Trace operator = 10.5\% \ \ Mrx. vec. mltp.  = 12.8\% \ Avr. \# singular val. = 48.8, \ nnz = 2,470,748   & \textbf{Total CPU time = 162.9} \ \ Preconditioner = 97.4\% \ nnz = 15,128,232 \\ 
 \hline
 $P(6,4,41,8.5\pi)$  \ \ \ \ \ \  (size$(\mathcal{A})=327180$) & 127 &  127  & \textbf{Total CPU time = 815.6}   \ \ Preconditioner = 37.4\% \ Truncation = 38.1\% \ \ \ \ \ \ \ \ \ Trace operator = 10.8\% \ \ Mrx. vec. mltp.  = 10.4\% \ Avr. \# singular val. = 98.0, \  nnz = 2,470,748   & \textbf{Total CPU time = 877.5} \ \ Preconditioner = 97.6\% \ nnz = 15,128,232\\ 
 \hline
\end{tabular}
\caption{Results of Test 2  showing total CPU times, total iterations, average number of singular values used per iteration and CPU times for different operations   from
preconditioned low-rank BICG (P-LRBICG) solver  compared with
those from plain preconditioned BICG (P-BICG) for selected parameters $P(N, Q, N_{p},c)$. size$(\mathcal{A})$ represents the size of the matrix $\mathcal{A}$.}
\label{table:2}
\end{center}
\end{table}

 \begin{table}
 \begin{center}
\begin{tabular}{ |p{3.1cm}|p{1.7cm}|p{1.5cm}| p{3.5cm} | p{3.4cm}| } 
\hline
    & P-LRBICG \ \ \ \ \ \ \  Iterations & P-BICG \ \ \ \ \ \ \ \ \ Iterations & P-LRBICG CPU time \& average number of singular values  & P-BICG CPU time \\ 
\hline
$P(5,3,41,1.5\pi)$ \ \ \ \ \ \  (size$(\mathcal{A})=327180$) & 8 & 8 & \textbf{Total CPU time = 11.9}   \ \ Preconditioner = 68.6\% \ Truncation = 17.3\% \ \ \ \ \ \ \ \ \ Trace operator = 3.6\% \ \ Mrx. vec. mltp. = 6.3\% \ Avr. \# of singular val. = 28.8, \  nnz = 2,459,356   & \textbf{Total CPU time = 17.5} \ \ Preconditioner = 96.0\% \  nnz = 3,411,268  \\ 
\hline
 $P(5,3,41,3.5\pi)$  \ \ \ \ \ \  (size$(\mathcal{A})=327180$) & 17 & 17 &  \textbf{Total CPU time = 25.8}   \ \ Preconditioner = 68.5\% \ Truncation = 17.9\% \ \ \ \ \ \ \ \ \ Trace operator = 4.3\% \ \ Mrx. vec. mltp.  = 6.9\% \ Avr. \# singular val. = 39.2, \  nnz = 2,459,356   & \textbf{Total CPU time = 33.7} \ \ Preconditioner = 97.4\% \  nnz = 3,411,268 \\ 
 \hline
 $P(5,3,41,8.5\pi)$  \ \ \ \ \ \  (size$(\mathcal{A})=327180$) & 67 &  63  & \textbf{Total CPU time = 105.8}   \ \ Preconditioner = 68.4\% \ Truncation = 18.2\% \ \ \ \ \ \ \ \ \ Trace operator = 4.4\% \ \ Mrx. vec. mltp.  = 6.3\% \ Avr. \# singular val. = 49.2, \  nnz = 2,459,356   & \textbf{Total CPU time = 134.1} \ \ Preconditioner = 97.2\% \  nnz = 3,411,268  \\ 
 \hline
\end{tabular}
\caption{Results of Test 2 showing total CPU times, total iterations, average number of singular values used per iteration and CPU times for different operations   from
preconditioned low-rank BICG (P-LRBICG) solver  compared with
those from plain preconditioned BICG (P-BICG) for selected parameters $P(N, Q, N_{p},c)$. size$(\mathcal{A})$ represents the size of the matrix $\mathcal{A}$.}
\label{table:3}
\end{center}
\end{table}

  \begin{table}[h!]
 \begin{center}
\begin{tabular}{ |p{3.1cm}|p{1.7cm}|p{1.5cm}| p{3.5cm} | p{3.4cm}| } 
\hline
    & P-LRBICG \ \ \ \ \ \ \  Iterations & P-BICG \ \ \ \ \ \ \ \ \ Iterations & P-LRBICG CPU time \& average number of singular values  & P-BICG CPU time \\ 
\hline
$P(7,5,21,1.5\pi)$ \ \ \ \ \ \  (size$(\mathcal{A})=299376$) & 12 & 12 & \textbf{Total CPU time = 60.2}   \ \ Preconditioner = 1.4\% \ Truncation = 59.8\% \ \ \ \ \ \ \ \ \ Trace operator = 19.7\% \ \ Mrx. vec. mltp. = 16.5\% \ Avr. \# of singular val. = 35.2, \  nnz = 159,974   & \textbf{Total CPU time = 74.5} \ \ Preconditioner = 96.4\% \  nnz = 15,472,776
  \\ 
\hline
 $P(7,5,21,3.5\pi)$  \ \ \ \ \ \  (size$(\mathcal{A})=299376$) & 40 & 40 &  \textbf{Total CPU time = 213.7}   \ \ Preconditioner = 1.5\% \ Truncation = 59.8\% \ \ \ \ \ \ \ \ \ Trace operator = 20.6\% \ \ Mrx. vec. mltp.  = 15.7\% \ Avr. \# singular val. = 53.8, \  nnz = 159,974   & \textbf{Total CPU time = 245.7} \ \ Preconditioner = 96.8\% \  nnz = 15,472,776
  \\ 
 \hline
 $P(7,5,21,8.5\pi)$  \ \ \ \ \ \  (size$(\mathcal{A})=299376$) & 133 &  133  & \textbf{Total CPU time = 915.8}   \ \ Preconditioner = 1.8\% \ Truncation = 61.9\% \ \ \ \ \ \ \ \ \ Trace operator =19.8\% \ \ Mrx. vec. mltp.  = 13.9\% \ Avr. \# singular val. = 91.8, \  nnz = 159,974   & \textbf{Total CPU time = 780.1} \ \ Preconditioner = 96.9\% \  nnz = 15,472,776
 \\ 
 \hline
\end{tabular}
\caption{Results of Test 2 showing total CPU times, total iterations, average number of singular values used per iteration and CPU times for different operations   from
preconditioned low-rank BICG (P-LRBICG) solver  compared with
those from plain preconditioned BICG (P-BICG) for selected parameters $P(N, Q, N_{p},c)$. size$(\mathcal{A})$ represents the size of the matrix $\mathcal{A}$.}
\label{table:4}
\end{center}
\end{table}

\section{Conclusion} \label{sec:conclusion} 
In this article, we considered the approximate solution to the stochastic Helmholtz problem using a low-rank approach.
An advantage of a low-rank approach is that it does does not require construction of the system matrix due to the Kronecker structure of the matrix. We have shown that when the operations in a low-rank approach are applied appropriately, it is possible to get an advantage in terms of CPU time. In this paper, we solved the indefinite system matrix (\ref{eqn:mainsystem}) using a preconditioned low-rank BICG method. We proposed a general approach to implement the preconditioner cheaply for the discretizations of the stochastic PDEs. Indefiniteness in the system matrix naturally leads to further difficulties. The total number of iterations required for the convergence increases, more singular values are used and hence  the low-rank operations become more expansive. However, even for indefinite problems, it is possible to save CPU time compared to the full rank approach. The application of the preconditioner increases memory requirements, however there is still a big advantage in terms of memory usage compared to the full rank approach.

\bibliographystyle{plain} 
\bibliography{ref}

\end{document}